\documentclass[a4paper,11pt]{article}
\usepackage{graphicx} 
\usepackage{xcolor}
\usepackage{amsmath,amssymb,amsthm}
\usepackage{mathtools}
\usepackage{booktabs}

\usepackage[top=38truemm,bottom=35truemm,left=35truemm,right=35truemm]{geometry}    

\theoremstyle{plain}
\newtheorem{theorem}{Theorem}
\newtheorem{proposition}[theorem]{Proposition}
\newtheorem{corollary}[theorem]{Corollary}

\theoremstyle{remark}

\newtheorem{remark}{Remark}

\theoremstyle{definition}
\newtheorem{definition}{Definition}

\title{Representation and Integration by Parts Formulas\\for Affine Processes}
\author{Arturo Kohatsu-Higa,$^{1}$ Yuma Tamura$^{2}$\\
\\
Department of Mathematical Sciences \\
College of Science and Engineering \\
Ritsumeikan University \\
1-1-1 Noji-higashi, Kusatsu, 525-8577, JAPAN \\
Email: khts00@fc.ritsumei.ac.jp,$^{1}$ ytamura11029@gmail.com$^{2}$}
\date{}

\begin{document}

\maketitle

\begin{abstract}
    Affine processes play an important role in mathematical finance and other applied areas due to their tractable structure. In the present article, we derive probabilistic representations and integration by parts (IBP) formulas for expectations involving affine processes. These formulas are expressed in terms of expectations of affine processes with modified parameters and are derived using Fourier analytic techniques and characteristic functions. Notably, our method does not require pathwise differentiability, allowing us to handle models with square-root diffusion coefficients for a large set of parameters. The methodology can be applied to the classic Cox--Ingersoll--Ross (CIR) model, a model for interest rates in mathematical finance, where the initial value derivative corresponds to one of the ``Greeks'' used in option pricing in mathematical finance. Furthermore, we illustrate the theory with an application to a population evolution model arising as a scaling limit of discrete branching processes. Our approach offers a unified and robust framework for sensitivity analysis in models where classical Malliavin calculus techniques are difficult to apply.
\end{abstract}


\section{Introduction}
Affine processes are widely used across applied fields, including but not limited to mathematical
finance. For references on the subject, we refer to \cite{Alfonsi}.
In the present article, we are interested in providing probabilistic representations for derivatives of expectations and the integration by parts formulas associated to these processes. We believe that these formulas are suitable for simulation and, potentially, for calibration, even when the payoff lacks smoothness. 

In more explicit terms, given an affine process $X^x$ starting at $x$ and a test function $f$, we are interested in obtaining representation formulas of the type 
\[
\partial_xE[f(X^x_t)]=E[g(\tilde{X}^x_t)].
\]
Here $g$ is a function which depends on $f$ and the parameters of the affine process. Similarly, $\tilde{X}^x$ represents a vector of affine process whose parameters are related but not necessarily the same as $X^x$. 

In particular, we remark that many available representations of the above type rely on pathwise differentiation, which is delicate or even infeasible for square-root diffusions. One of these cases corresponds to affine processes where the diffusion coefficient is the square root function. Its path differentiation has been previously obtained in the one dimensional case under certain restrictive parameter restrictions (see e.g. \cite{Alos}).  For Bessel processes we refer the reader to \cite{Altman}.
Here, using a different technique based on characteristic functions, we obtain slightly more general results that hints to a generalization of the concept of differentiation.

In many situations, one is also interested in differentiation with respect to various parameters. In order to explain the ideas in simple terms, we concentrate on the derivatives with respect to $x$. Other examples are also considered but the methodology is the same with some additional considerations. For example, we also compute for simplicity the derivative with respect to $\beta $.

In Section \ref{sec:pre}, we provide the general set-up of our work. In Section \ref{sec:3}, we provide first the characteristic function characterization of affine processes and derive our main results concerning the derivative of expectations of functionals of affine processes and their integration by parts. 
In Section \ref{sec:3.2}, we consider as an example, derivatives with respect to one of the parameters in the affine model. In Section \ref{sec:4}, we consider a simple multi-dimensional generalization and finally in Section \ref{sec:5}, we consider an application to the CIR model in finance and an application to a limit model of population evolution.

\section{Preliminaries}
\label{sec:pre}
We use the following notation for sets of numbers:
\begin{align*}
    \mathbb{R} &:= \text{(The set of all real numbers)},  \\
    \mathbb{R}_{\ge0} &:= \{ x \in \mathbb{R} \mid x \ge 0 \},  \\
    \mathbb{R}_{>0} &:= \{ x \in \mathbb{R} \mid x > 0 \},  \\
    \mathbb{C} &:= \text{(The set of all complex numbers)}.
\end{align*}
We also set
\[
    D := \mathbb{R}_{>0} \times \mathbb{R} \times \mathbb{R}_{\ge0}.
\]
In addition, we use the following notation for spaces of functions:
\begin{align*}
    L^1(\mathbb{R}) &:= \text{(The set of all Lebesgue-integrable functions on \(\mathbb{R}\))}, \\
    L^1( \Omega ) &:= \text{(The set of all random variables which have finite expectation)}.
\end{align*}

Now, we introduce \emph{affine processes}. Roughly speaking, one-dimensional affine processes are defined as follows:
\begin{definition}
    A one-dimensional Markov process $(X^x)_x$, where \(x\) represents its starting point, is called \emph{affine} if for every $ t \in \mathbb{R}_{>0} $, the characteristic function $ \theta \mapsto E[ \exp( i \theta X^x_t ) ] \colon \mathbb{R} \to \mathbb{C} $ has exponential-affine dependence on $x$. More precisely, if there exist functions $ g,h \colon [0,\infty) \times \mathbb{R} \to \mathbb{C} $ such that $ E[ \exp( i \theta X^x_t ) ] $ has the following form:
    \begin{align}
    \label{eq:cf}
      E[ \exp( i \theta X^x_t ) ] = \exp( g(t,\theta) + x h(t,\theta) ).  
    \end{align}
\end{definition}
We do not reproduce the full general definition, as we focus on diffusion-type affine processes; for an exact definition, see \cite{DFS2003} and \cite{Alfonsi} for instance.

Now, for fixed $ ( \alpha, \beta, b ) \in D $, consider the following one-dimensional stochastic differential equation (SDE) of \( \mathbb{R}_{\ge0} \)-valued process for \( x \in \mathbb{R}_{\ge0} \):
\begin{equation}\label{1dimAffineSDE}
    dX^x_t = \sqrt{ \alpha X^x_t } \,dW_t + ( \beta X^x_t + b )dt,\quad X^x_0=x.
\end{equation} 

This SDE admits a unique solution in law and it is known that the family $ (X^x)_x $ is an affine process (See e.g. \cite{Alfonsi}, Section 1.3). Thus we call the solution of the SDE (\ref{1dimAffineSDE}) \emph{affine diffusion with parameter $ ( \alpha, \beta, b ) $}.

Although not directly related to the main discussion, the behavior at the boundary \( 0 \) is as follows (See e.g. \cite{JYC2009}, Chapter 6).

(i) If \( 2b / \alpha = 0 \) i.e. \( b=0 \), \( X^0_t = 0 \) for all \(t\) a.s. For \( x \in \mathbb{R}_{>0} \), \( X^x \) hits \( 0 \) in finite time a.s. if \( \beta >0 \) and hits \( 0 \) with probability \( p \in (0,1) \) if \( \beta \le 0 \). In both cases, \(X^x\) is absorbed at the boundary \( 0 \). That is, if \( X^x_t(\omega) = 0 \), then \( X^x_s(\omega) = 0 \) for all \( s > t \).

(ii) If \( 0 < 2b / \alpha < 1 \), \( X^x \) hits \( 0 \) in finite time a.s. for all \( x \in \mathbb{R}_{\ge0} \) and reflects instantaneously. That is, the Lebesgue measure of \( \{ t \ge 0 \mid X^x_t(\omega) = 0 \} \) is \( 0 \) for almost all \( \omega \) for all \( x \in \mathbb{R}_{\ge0} \).

(iii) If \( 2b / \alpha \ge 1 \), \( X^x \) never hits \( 0 \) (except for the particular case that the process starts at \( x=0 \) and we consider the time \(t=0\)) a.s. for all \( x \in \mathbb{R}_{\ge0} \).

In the proofs to follow, we use Fourier transforms. For \( f \in L^1(\mathbb{R}) \), we denote its Fourier transform by \( \widehat{f} \), and adopt the following definition:
\[
    \widehat{f}(\theta) := \int_{\mathbb{R}} f(y) e^{ - i \theta y } \,dy, \quad \theta \in \mathbb{R}.
\]
In this case, if \( \widehat{f} \in L^1(\mathbb{R}) \), it holds that
\[
    f(y) = \frac{1}{2\pi} \int_{\mathbb{R}} \widehat{f}(\theta) e^{ i \theta y } \,d\theta, \quad y\in\mathbb{R}.
\]

\section{Derivatives of expectations of affine diffusions}
\label{sec:3}
It is known that the function $ h $ in the characteristic function \eqref{eq:cf} corresponding to the SDE (\ref{1dimAffineSDE}) satisfies the following ordinary differential equation (See \cite{DFS2003}, Theorem 2.7 for example):
\begin{equation}\label{1dimODE}
    \partial_t h(t,\theta) = \frac{ \alpha }{2} h(t,\theta)^2 + \beta h(t,\theta), \qquad h(0,\theta) = i \theta.
\end{equation}
Furthermore, $ g $ is given by the integral of $h$:
\[
    g(t,\theta) = \int_0^t b h( s,\theta ) \,ds.
\]
In this case, we can solve the differential equation (\ref{1dimODE}) explicitly, and the solution is
\[
    h(t,\theta) = \begin{dcases}
        \frac{ 2 i \theta }{ 2 - \alpha i \theta t }; &\text{ if } \beta = 0,\\
        \frac{ 2 \beta e^{ \beta t } i \theta }{ 2 \beta - ( e^{ \beta t } -1 ) \alpha i \theta };&\text{ if } \beta \neq 0.
    \end{dcases}
    \]
Moreover,
\[
    g(t,\theta) = 
    \begin{dcases}
       \frac{ 2b }{ \alpha } \mathop{\mathrm{Log}} \frac{2}{ 2 - \alpha i \theta t }; &\text{ if } \beta = 0,\\
       \frac{2b}{\alpha} \mathop{\mathrm{Log}} \frac{ 2\beta }{ 2 \beta - ( e^{ \beta t } - 1 ) \alpha i \theta };
       &\text{ if } \beta \neq  0.
    \end{dcases}
\]
Here, \( \mathrm{Log} \) denotes the principal value of the complex logarithm, that is, the imaginary part of \( \mathop{\mathrm{Log}} z \) lies in the interval \( ( -\pi, \pi ] \).

Thus, the characteristic function of \( X^x_t \) is given as follows.
\[
    E[ \exp( i \theta X^x_t ) ] = 
    \begin{dcases}
      \biggl( \frac{2}{ 2 - \alpha i \theta t } \biggr)^{2b/\alpha} \exp\biggl( \frac{ 2 i \theta }{ 2 - \alpha i \theta t } x \biggr);  &\text{ if } \beta = 0,\\
      \biggl( \frac{ 2\beta }{ 2 \beta - ( e^{ \beta t } - 1 ) \alpha i \theta } \biggr)^{2b/\alpha} \exp\biggl( \frac{ 2 \beta e^{ \beta t } i \theta }{ 2 \beta - ( e^{ \beta t } -1 ) \alpha i \theta } x \biggr);
      &\text{ if } \beta \neq 0.
    \end{dcases}
\]
Hereafter, throughout this paper, for a complex number \( \alpha \), its real power \( \alpha^c \) is understood to represent \( e^{ c \mathop{\mathrm{Log}} \alpha } \).
Note that the limit of \( E[ \exp( i \theta X^x_t ) ] \) as \( \beta \to 0 \) coincides with the case \( \beta = 0 \). That is, \( \beta \mapsto E[ \exp( i \theta X^x_t ) ] \) is continuous at \(0\).

\subsection{Derivative with respect to the initial value}

Now, our first theorem is about the derivative of $ E[ f(X^x_t) ] $ with respect to the initial value $x$.
\begin{theorem}\label{delta}
    Let $ t \in \mathbb{R}_{>0} $ and \( (\alpha,\beta,b) \in D \). Moreover, let $ X^x $ and $ X^{1,x} $ be affine diffusions with parameters $ (\alpha,\beta,b) $ and $ (\alpha,\beta,b+\alpha/2) $ respectively and initial value $ x \in \mathbb{R}_{>0} $. Then, for $ f \in L^1(\mathbb{R}) $ such that \( \widehat{f}\in L^1(\mathbb{R}) \) and $ f(X^x_t), f(X^{1,x}_t)\in L^1(\Omega) $, it holds that
    \[
        \partial_x E[ f( X^x_t ) ] =
        \begin{dcases}
            \frac{2}{ \alpha t } \bigl( E[ f( X^{1,x}_t) ] - E[ f(  X^x_t  ) ] \bigr);    & \text{ if }  \beta=0,    \\
            \frac{ 2\beta e^{ \beta t } }{ \alpha ( e^{ \beta t } -1 ) } \bigl( E[ f( X^{1,x}_t) ] - E[ f( X^x_t ) ] \bigr);  &\text{ if }  \beta \neq 0.
        \end{dcases}
    \]
\end{theorem}
\begin{proof}
    In the following, we denote the characteristic function of \( X^x_t \) (resp. \( X^{1,x}_t \)) by \( \varphi^x_t \) (resp. \(\varphi^{1,x}_t\)). Moreover, we denote the law of \( X^x_t \) (resp. \( X^{1,x}_t \)) by \( \nu^x_t \) (resp. \(\nu^{1,x}_{t}\)). Then, by using inverse Fourier transform,
    \begin{align}
        E[ f( X^x_t ) ] &= \int_{\mathbb{R}} f(y) \,\nu^x_t(dy) \notag\\
        &= \int_{\mathbb{R}} \frac{1}{2\pi} \int_{\mathbb{R}} \widehat{f}(\theta) e^{ i \theta y } \,d\theta \,\nu^x_t(dy)  \notag\\
        &= \frac{1}{2\pi} \int_{\mathbb{R}} \widehat{f}(\theta) \int_{\mathbb{R}} e^{ i \theta y } \,\nu^x_t(dy) \,d\theta  \label{characteristic_expression}\\
        &= \frac{1}{2\pi} \int_{\mathbb{R}} \widehat{f}(\theta) \, \varphi^x_t(\theta) \, d \theta.  \notag
    \end{align}
    Here, we have applied Fubini's theorem as \( (\theta,y) \mapsto \widehat{f}(\theta) e^{ i \theta y } \) is \( ( \mathrm{Leb} \times \nu^x_t ) \)-integrable by assumption, where \( \mathrm{Leb} \) denotes the Lebesgue measure.
    
    First, we consider the case \( \beta = 0 \). In this case,
    \[
        \varphi^x_t(\theta) = \biggl( \frac{2}{ 2 - \alpha i \theta t } \biggr)^{2b/\alpha} \exp\biggl( \frac{ 2 i \theta }{ 2 - \alpha i \theta t } x \biggr).
    \]
    Therefore,
    \[
        \partial_x \varphi^x_t(\theta) = \biggl( \frac{2}{ 2 - \alpha i \theta t } \biggr)^{2b/\alpha+1} i \theta \exp\biggl( \frac{ 2 i \theta }{ 2 - \alpha i \theta t } x \biggr).
    \]
    Moreover, we have
    \begin{align*}
        | \partial_x \varphi^x_t(\theta) | & \le \biggl| \frac{ 2 i \theta }{ 2 - \alpha i \theta t } \biggr| | \varphi^x_t(\theta) | \\
        &= \frac{ 2 | \theta | }{ \sqrt{ 4 + \alpha^2 \theta^2 t^2 } } | \varphi^x_t(\theta) | \\
        &\le \frac{ 2 | \theta | }{ \alpha t | \theta | } \cdot 1 \\
        &= \frac{2}{ \alpha t }.
    \end{align*}
    for \( \theta \ne 0 \). Thus,
    \[
        | \widehat{f}(\theta) \, \partial_x \varphi^x_t(\theta) | \le \frac{ 2 }{ \alpha t } | \widehat{f}(\theta) |,
    \]
    where \( \widehat{f}(\theta) \) is \( \mathrm{Leb} \)-integrable and does not depend on \(x\). Then, the mean value theorem and the dominated convergence theorem yield that
    \[
        \partial_x E[ f( X^x_t ) ] = \frac{1}{2\pi} \int_{\mathbb{R}} \widehat{f}(\theta) \partial_x \varphi^x_t(\theta) \, d \theta.
    \]
    Also, note that (\ref{characteristic_expression}) is true for \( X^{1,x}_t \), namely,
    \[
        E[ f( X^{1,x}_t ) ] = \frac{1}{2\pi} \int_{\mathbb{R}} \widehat{f}(\theta) \varphi^{1,x}_t(\theta) \, d \theta.
    \]
%

    We also have the relation
    \[
        i \theta = \frac{2}{ \alpha t } \biggl( 1 - \frac{ 2 - \alpha i \theta t }{2} \biggr),
    \]
    which implies that
    \begin{align*}
        \partial_x \varphi^x_t(\theta) &= \biggl( \frac{2}{ 2 - \alpha i \theta t } \biggr)^{2b/\alpha+1} i \theta \exp\biggl( \frac{ 2 i \theta }{ 2 - \alpha i \theta t } x \biggr)   \\
        &= \biggl( \frac{2}{ 2 - \alpha i \theta t } \biggr)^{2b/\alpha+1} \frac{2}{ \alpha t } \biggl( 1 - \frac{ 2 - \alpha i \theta t }{2} \biggr) \exp\biggl( \frac{ 2 i \theta }{ 2 - \alpha i \theta t } x \biggr)   \\
        &= \frac{2}{ \alpha t } \biggl( \biggl( \frac{2}{ 2 - \alpha i \theta t } \biggr)^{2b/\alpha+1} \exp\biggl( \frac{ 2 i \theta }{ 2 - \alpha i \theta t } x \biggr) - \biggl( \frac{2}{ 2 - \alpha i \theta t } \biggr)^{2b/\alpha} \exp\biggl( \frac{ 2 i \theta }{ 2 - \alpha i \theta t } x \biggr) \biggr)   \\
        &= \frac{2}{ \alpha t } ( \varphi^{1,x}_t( \theta ) - \varphi^x_t(\theta) ).
    \end{align*}
    This completes the argument for the case where \( \beta = 0 \).

    For the case where \( \beta \neq 0 \), the argument is similar to the one above, using the fact that
    \[
        i \theta = \frac{ 2 \beta }{ \alpha ( e^{ \beta t } - 1 ) } \biggl( 1 - \frac{ 2 \beta - ( e^{ \beta t } - 1 ) \alpha i \theta }{ 2 \beta } \biggr).
    \]
    Hence, the proof is omitted.
\end{proof}

\begin{remark}
The restrictions on the function $f$ in the above result can be weakened in different directions using approximations based on the regularity properties of the underlying processes. We do not enter into these technicalities in order to make our exposition simple. Just as a note, it is well known that if \( f \) is rapidly decreasing, then \( \widehat{f} \) is also rapidly decreasing and, in particular, integrable. Therefore using approximation functions within this class one may extend the class of functions $f$ for which the above result applies.
\end{remark}

\begin{remark}Obtaining the above formulas under different conditions surfaced in the process of this research. For example, the assumption ``\( \widehat{f} \) is integrable'' in the above Theorem can be dropped if we assume instead that \( 2b > \alpha \), which implies the integrability of both \( \varphi^x_t \) and \( \varphi^{1,x}_t \). {Also other formulations based on Plancherel theorem are possible which requires that either the densities of the underlying processes or their characteristic functions  are square integrable.}
\end{remark}

\begin{remark}
    As in the case of \( g \) and \( h \), the formula in Theorem \ref{delta} also agrees with the case \( \beta = 0 \) when we take the limit as \( \beta \to 0 \). The same holds for the subsequent results.
\end{remark}

Moreover, the following formula can also be obtained.

\begin{theorem}\label{IBP}
    Let $ t \in \mathbb{R}_{>0} $ and \( (\alpha,\beta,b-\alpha/2) \in D \). Moreover, let $ X^x $ and $ X^{-1,x} $ be affine diffusions with parameters $ (\alpha,\beta,b) $ and $ (\alpha,\beta,b-\alpha/2) $ respectively, and their initial values are both $ x \in \mathbb{R}_{>0} $. Then, for $ f \in C^1(\mathbb{R}) $ such that $ f $, $ f^\prime, \widehat{f},\widehat{f^\prime} \in L^1(\mathbb{R}) $ and $ f^\prime(X^x_t) $, $ f(X^x_t), f(X^{-1,x}_t) \in L^1( \Omega ) $, it holds that
    \[
        E[ f^\prime( X^x_t ) ] =
        \begin{dcases}
            \frac{2}{ \alpha t } \bigl( E[ f(X^x_t) ] - E[ f( X^{-1,x}_t ) ] \bigr);    &  \text{ if }  \beta=0,    \\
            \frac{ 2\beta }{ \alpha ( e^{ \beta t } -1 ) } \bigl( E[ f(X^x_t) ] - E[ f( X^{-1,x}_t ) ] \bigr);  & \text{ if } \beta \neq 0.
        \end{dcases}
    \]
\end{theorem}
%
\begin{proof}
    First of all, we have
    \begin{align*}
        E[ f^\prime( X^x_t ) ] &= \int_{\mathbb{R}} f^\prime(y) \, \nu^x_t(dy) \\
        &= \int_{\mathbb{R}} \frac{1}{2\pi} \int_{\mathbb{R}} \widehat{f^\prime}(\theta) e^{ i \theta y } \,d\theta \, \nu^x_t(dy)    \\
        &= \frac{1}{2\pi} \int_{\mathbb{R}} \int_{\mathbb{R}} \widehat{f^\prime}(\theta) e^{ i \theta y }  \, \nu^x_t(dy) \,d\theta    \\
        &= \frac{1}{2\pi} \int_{\mathbb{R}} i \theta \widehat{f}(\theta) \varphi^x_t(\theta) \,d\theta \\
        &= \frac{1}{2\pi} \int_{\mathbb{R}} \widehat{f}(\theta) i \theta \varphi^x_t(\theta) \,d\theta.
    \end{align*}
    
    By an argument analogous to that in the proof of Theorem \ref{delta}, we obtain the proof using the following relations
    \begin{align*}
        i \theta = &\frac{2}{ \alpha t } \biggl( 1 - \frac{ 2 - \alpha i \theta t }{2} \biggr);\quad \text{ if }  \beta = 0,\\
         i \theta = &\frac{ 2 \beta }{ \alpha ( e^{ \beta t } - 1 ) } \biggl( 1 - \frac{ 2 \beta - ( e^{ \beta t } - 1 ) \alpha i \theta }{ 2 \beta } \biggr); \quad \text{ if }  \beta \neq 0 .
    \end{align*}
\end{proof}

\begin{remark}
    The assumption ``\( \widehat{f} \) and \( \widehat{f^\prime} \) are integrable'' in the above Theorem can be dropped if we assume instead that \( b > \alpha \), which implies the integrability of the characteristic functions of \( X^x_t \) and \( X^{-1,x}_{t} \).
\end{remark}

\begin{remark}
    The assumption \( b - \alpha/2 \ge 0 \) is equivalent to the condition for \( X^x \) never to hit \(0\). See Section \ref{sec:pre}.
\end{remark}

\begin{remark}
    The above two theorems can be summarized in the following table.
    \begin{table}[htb]
    \centering
    \caption{The summary of Theorems \ref{delta} and \ref{IBP}}\label{xxx}%
    \begin{tabular}{@{}lll@{}}
    \toprule
    Case & Shift & Identity (schematic)\\
    \midrule
    \( \partial_x \) (Thm. \ref{delta}) & \( b \leadsto b + \alpha / 2 \) & \( \partial_x E[f(X^x_t)] = C_1( E[f( X^{1,x}_t )] - E[f(X^x_t)] ) \) \\
    IBP (Thm. \ref{IBP}) &\( b \leadsto b - \alpha / 2 \) &  \( E[f^\prime(X^x_t)] = C_2( E[f( X^{x}_t )] - E[f(X^{-1,x}_t)] ) \) \\
    \bottomrule
    \end{tabular}
    \end{table}

\end{remark}

Also, combining these two theorems, we have the following formula.

\begin{corollary}\label{affinecombined}
     Let $ t \in \mathbb{R}_{>0} $ and \( (\alpha,\beta,b) \in D \). Moreover, let $ X^x $ and $ X^{1,x} $ be affine diffusions with parameters $ (\alpha,\beta,b) $ and $ (\alpha,\beta,b+\alpha/2) $ respectively, and their initial values are $ x \in \mathbb{R}_{>0} $. Then, for $ f \in C^1(\mathbb{R}) $ such that $ f $, $ f^\prime $, \( \widehat{f}, \widehat{f^\prime} \in L^1(\mathbb{R}) \) and $ f^\prime( X^{1,x}_t) $, $ f(X^x_t), f(X^{1,x}_t) \in L^1(\Omega) $, it holds that
    \[
        \partial_x E[ f( X^x_t ) ] = e^{ \beta t } E[ f^\prime( X^{1,x}_t ) ].
    \]
\end{corollary}
\begin{remark}
    In fact, H. Totsuka (Weak derivatives of Squared Bessel processes, unpublished master's thesis, 2022, Ritsumeikan University) derived formulas similar to Theorems \ref{delta}, \ref{IBP} and Corollary \ref{affinecombined} for the squared Bessel process instead of affine diffusions. Note that the squared Bessel process of dimension \(\delta\) is an affine diffusion with parameters \( (4,0,\delta) \). Essentially, our result is an extension of hers, though not a complete one. This is because, while we used characteristic functions and the Fourier transform, she based her proof on the density functions. Instead of imposing integrability on the Fourier transform of the test function, she imposed a stronger integrability condition on the test function itself.
\end{remark}

\subsection{Other derivatives}
\label{sec:3.2}
We believe that derivatives with respect to other parameters can be computed with similar ideas, although each example requires particular manipulations. Just as an example, we give here the formulas for the differentiation with respect to $\beta$. The result is as follows.

\begin{proposition}
    Let $ t \in \mathbb{R}_{>0} $ and \( (\alpha,\beta,b) \in D \). Moreover, let $ X^{i,x} $ be affine diffusions with parameters $ (\alpha,\beta, b + i \alpha / 2) $ for \( i=0,1,2 \), and all of them have initial value $ x \in \mathbb{R}_{>0} $. Then, for $ f \in L^1(\mathbb{R}) $ such that \( \widehat{f}\in L^1(\mathbb{R}) \) and $ f(X^{i,x}_t) \in L^1(\Omega)$, \( i=0,1,2 \), it holds for \( \beta \neq 0 \) 
    \begin{align}
        \partial_\beta E[ f( X^{0,x}_t ) ] = & \frac{ 2 x e^{ \beta t } ( \beta t - ( e^{ \beta t } - 1 ) ) }{ \alpha ( e^{ \beta t } - 1 )^2 } E[ f( X^{2,x}_t ) ] \notag  \\
        &+ \frac{ 2 }{ \alpha ( e^{ \beta t } - 1 ) } \biggl( {b} \frac{ \beta t e^{ \beta t }  - ( e^{ \beta t } - 1 )  }{ \beta } + \beta t x e^{ \beta t } \biggr) E[ f( X^{1,x}_t ) ] \notag \\
        &- \frac{ 2 ( \beta t e^{ \beta t }  - ( e^{ \beta t } - 1 ) ) }{ \alpha ( e^{ \beta t } - 1 ) } \biggl( \frac{ b }{ \beta } + \frac{ x e^{ \beta t }  }{ e^{ \beta t } - 1 } \biggr) E[ f( X^{0,x}_t ) ] .\label{beta_E}
    \end{align}
    Otherwise, in the case that $\beta=0$, we have
    \begin{align*}
        \Bigl. \partial_\beta E[ f( X^{0,x}_t ) ] \Bigr|_{\beta=0} = - \frac{x}{\alpha} E[ f( X^{2,x}_t ) ] + \frac{2}{\alpha} \biggl( \frac{bt}{2} - x \biggr) E[ f( X^{1,x}_t ) ] + \frac{1}{\alpha} ( bt - x ) E[ f( X^{0,x}_t ) ].
    \end{align*}
\end{proposition}
\begin{proof}
    Since the argument is essentially the same as that in Theorem \ref{delta}, we present only the calculation for this proposition assuming that all appearing quantities are integrable. First, assume \( \beta \neq 0 \) and let
    \[
        \psi(\beta) := \frac{ 2 \beta }{ 2 \beta - ( e^{ \beta t } - 1 ) \alpha i \theta }.
    \]
    
    Then recall that the characteristic function of \( X^{0,x}_t \) is as follows:
    \[
        \varphi^{\alpha,\beta,b}_{x,t}(\theta) := \psi(\beta)^{2b/\alpha} \exp( \psi(\beta) e^{ \beta t } i \theta x ).
    \]
    Therefore, we get
    \begin{align}
        &\partial_\beta \varphi^{\alpha,\beta,b}_{x,t}(\theta) \notag \\
        &= \frac{2b}{\alpha} \psi(\beta)^{2b/\alpha-1} \psi^\prime(\beta) \exp( \psi(\beta) e^{ \beta t } i \theta x ) +  \psi(\beta)^{2b/\alpha} \exp( \psi(\beta) e^{ \beta t } i \theta x ) i \theta x ( \psi^\prime(\beta) e^{ \beta t } + t \psi(\beta) e^{ \beta t } ) \notag \\
        &= \biggl( \frac{2b}{\alpha} \psi(\beta)^{-1} \psi^\prime(\beta) +  \psi^\prime(\beta) i \theta x e^{ \beta t } + t \psi(\beta) i \theta x e^{ \beta t } \biggr) \psi(\beta)^{2b/\alpha}\exp( \psi(\beta) e^{ \beta t } i \theta x ). \label{beta}
    \end{align}
    
    By direct differentiation, we have 
    \[
        \psi^\prime(\beta) = \frac{ \beta t e^{ \beta t } \alpha - ( e^{ \beta t } - 1 ) \alpha }{ 2 \beta^2 } \psi(\beta)^2 i \theta.
    \]
    Using the decomposition 
    \[
        i \theta = \frac{ 2 \beta }{ \alpha ( e^{ \beta t } - 1 ) } - \frac{ 2 \beta }{ \alpha ( e^{ \beta t } - 1 ) \psi(\beta) },
    \]
    (\ref{beta}) can be rewritten as follows:
    \begin{align*}
        \partial_\beta \varphi^{\alpha,\beta,b}_{x,t}(\theta) &= \frac{ 2 x e^{ \beta t } }{ \alpha ( e^{ \beta t } - 1 ) } \biggl( \frac{ \beta t e^{ \beta t }  - ( e^{ \beta t } - 1 )  }{ e^{ \beta t } - 1 } - \beta t \biggr) \psi(\beta)^{2b/\alpha+2} \exp( \psi(\beta) e^{ \beta t } i \theta x )   \\
        &+ \frac{ 2 }{ \alpha ( e^{ \beta t } - 1 ) } \biggl( {b} \frac{ \beta t e^{ \beta t }  - ( e^{ \beta t } - 1 )  }{ \beta } + \beta t x e^{ \beta t } \biggr) \psi(\beta)^{2b/\alpha+1} \exp( \psi(\beta) e^{ \beta t } i \theta x )   \\
        &- \frac{ 2 ( \beta t e^{ \beta t }  - ( e^{ \beta t } - 1 ) ) }{ \alpha ( e^{ \beta t } - 1 )^2 } \biggl( \frac{ b }{ \beta } + { x e^{ \beta t }  } \biggr) \psi(\beta)^{2b/\alpha} \exp( \psi(\beta) e^{ \beta t } i \theta x ).
    \end{align*}
    That is,
    \begin{align*}
        \partial_\beta \varphi^{\alpha,\beta,b}_{x,t}(\theta)     &= \frac{ 2 x e^{ \beta t } }{ \alpha ( e^{ \beta t } - 1 ) } \biggl( \frac{ \beta t e^{ \beta t }  - ( e^{ \beta t } - 1 )  }{ e^{ \beta t } - 1 } - \beta t \biggr) \varphi^{\alpha,\beta,b+\alpha}_{x,t}(\theta)   \\
        &+ \frac{ 2 }{ \alpha ( e^{ \beta t } - 1 ) } \biggl( {b} \frac{ \beta t e^{ \beta t }  - ( e^{ \beta t } - 1 )  }{ \beta } + \beta t x e^{ \beta t } \biggr) \varphi^{\alpha,\beta,b+\alpha/2}_{x,t}(\theta)   \\
        &- \frac{ 2 ( \beta t e^{ \beta t }  - ( e^{ \beta t } - 1 ) ) }{ \alpha ( e^{ \beta t } - 1 ) } \biggl( \frac{ b }{ \beta } + \frac{ x e^{ \beta t }  }{ e^{ \beta t } - 1 } \biggr) \varphi^{\alpha,\beta,b}_{x,t}(\theta).
    \end{align*}
    In this case, the coefficients of three characteristic functions are all continuous with respect to \(\beta\) on \( \mathbb{R} \setminus \{ 0 \} \), and are therefore bounded on any finite interval in \( \mathbb{R} \setminus \{ 0 \} \). Thus, using the mean value theorem and the dominated convergence theorem, together with the fact that the absolute value of a characteristic function is equal to \(1\), we obtain the conclusion in the case \( \beta \neq 0 \).
    
    Moreover, \( \beta \mapsto E[ f( X^{0,x}_t ) ] \) is continuous at \(0\) and the right hand side of (\ref{beta_E}) has a finite limit as \( \beta \to 0 \). Therefore, the result for $\beta=0$ follows.
\end{proof}
\section{An extension to particularly simple multidimensional case}
\label{sec:4}
Consider the following \( (\mathbb{R}_{\ge0})^d \)-valued SDE:
\begin{equation}\label{MultidimAaffineSDE}
    dX^x_t = \alpha(X^x_t) dW_t + ( \beta X^x_t + b )dt, \qquad X^x_0 = x,
\end{equation}
where $ W $ is a $d$-dimensional standard Brownian motion and $ \alpha \colon \mathbb{R}^d \to \mathbb{R}^d \times \mathbb{R}^d $ is defined by
\begin{equation*}
    \alpha((x_k)_{k=1,\dots,d})_{i,j} = \delta_{i,j} \sqrt{ \alpha_i x_i }, \qquad 1 \le i,j \le d
\end{equation*}
with $ \alpha_i > 0$, $ i = 1, \dots, d $, $ \beta = \mathrm{diag}(\beta_1, \dots, \beta_d) $ is a diagonal real matrix of order $d$, and $ b = (b_i)_{i=1,\dots,d} \in (\mathbb{R}_{\ge0})^d $. Also, we assume the initial value $ x = (x_i)_{i=1,\dots,d} $ satisfies $ x_i \ge 0 $, $ i = 1, \dots, d $. In this case, all components of $ ( (X^x)_1, \dots, (X^x)_d ) $ are independent, so $ X^x $ can essentially be regarded as a combination of one-dimensional affine diffusions. 

In fact, the solution \( (X^x)_{x} \) can be considered as a particularly simple \(d\)-dimensional affine diffusion. For general multidimensional affine diffusion, see \cite{DFS2003} for instance. As a consequence of the theorems stated in the previous section, we obtain the following results. In the statement we use $ e_i $ for the $i$-th vector of the standard basis of $\mathbb{R}^d$. Also, \( \widehat{f} \) denotes the usual Fourier transform on \( \mathbb{R}^d \) in this section.

\begin{theorem}\label{multidimaffinedelta}
    Let $ k $ be in $ \{ 1, \dots, d \} $, $ t \in \mathbb{R}_{>0} $ and \( ( \alpha_i,\beta_i,b_i ) \in D \) for all \(i\). Moreover, let $ X^x $ and $ X^{e_k,x} $ be the solutions of SDE (\ref{MultidimAaffineSDE}) whose parameters are $ ( (\alpha_i)_{i=1,\dots,d},\beta,b) $ and $ ( (\alpha_i)_{i=1,\dots,d},\beta, b + ( \alpha_k / 2 ) e_k ) $ respectively, and their initial values are $ x \in (0,\infty)^d $. Then, for $ f \in L^1(\mathbb{R}^d) $ such that \( \widehat{f} \in L^1( \mathbb{R}^d ) \) and $ f(X^x_t), f(X^{e_k,x}_t) \in L^1( \Omega )$, it holds that
    \[
        \partial_{x_k} E[ f( X^x_t ) ] =
        \begin{dcases}
            \frac{2}{ \alpha_k t } \bigl( E[ f( X^{e_k,x}_t) ] - E[ f(  X^x_t  ) ] \bigr);    & \text{ if }   \beta_{k,k}=0,    \\
            \frac{ 2 \beta_{k,k} e^{ \beta_{k,k} t } }{ \alpha_k ( e^{ \beta_{k,k} t } -1 ) } \bigl( E[ f( X^{e_k,x}_t) ] - E[ f( X^x_t ) ] \bigr);  & \text{ if } \beta_{k,k} \neq 0.
        \end{dcases}
    \]
\end{theorem}

The assumptions on \(f\) in Theorem \ref{multidimaffineIBP} and Corollary \ref{multidimaffinecombined} are as follows.

\begin{enumerate}
    \item $ f, \widehat{f} \in L^1( \mathbb{R}^d ) $,
    \item $ f(y_1,\dots,y_{k-1},\cdot,y_{k+1},\dots,y_d) \in C^1(\mathbb{R}) $ for each $ y_1,\dots,y_{k-1},y_{k+1},\dots,y_d \in \mathbb{R} $,
    \item $ \partial_{y_k} f \in L^1( \mathbb{R}^d ) $,
    \item $ \partial_{y_k} f ( X^x_t ), f ( X^x_t ), f ( X^{\ast,x}_t ) \in L^1( \Omega ) $.
\end{enumerate}
Here, \( X^{\ast,x}_t \) refers to \( X^{-e_k,x}_t \) in Theorem~\ref{multidimaffineIBP}, and to \( X^{e_k,x}_t \) in Corollary~\ref{multidimaffinecombined}.

\begin{theorem}\label{multidimaffineIBP}
    Let $ k $ be in $ \{ 1, \dots, d \} $, $ t \in \mathbb{R}_{>0} $, \( ( \alpha_i,\beta_i,b_i ) \in D \) for all \(i\) and \( ( \alpha_k,\beta_k,b_k-\alpha_k/2 ) \in D \). Moreover, let $ X^x $ and $ X^{-e_k,x} $ be the solutions of SDE (\ref{MultidimAaffineSDE}) whose parameters are $ ( (\alpha_i)_{i=1,\dots,d},\beta,b) $ and $ ( (\alpha_i)_{i=1,\dots,d},\beta, b - ( \alpha_k / 2 ) e_k ) $ respectively, and their initial values are $ x \in (0,\infty)^d $. Furthermore, assume the above assumptions on \(f\). Then, it holds that
    \[
        E[ (\partial_{y_k}f)( X^x_t ) ] =
        \begin{dcases}
            \frac{2}{ \alpha_k t } \bigl( E[ f(X^x_t) ] - E[ f( X^{-e_k,x}_t ) ] \bigr);    &  \text{ if }  \beta_{k,k}=0,    \\
            \frac{ 2\beta_{k,k} }{ \alpha_k ( e^{ \beta_{k,k} t } -1 ) } \bigl( E[ f(X^x_t) ] - E[ f( X^{-e_k,x}_t ) ] \bigr);  & \text{ if } \beta_{k,k} \neq 0.
        \end{dcases}
    \]
\end{theorem}
\begin{corollary}\label{multidimaffinecombined}
    Let $ k $ be in $ \{ 1, \dots, d \} $, $ t \in \mathbb{R}_{>0} $ and \( ( \alpha_i,\beta_i,b_i ) \in D \) for all \(i\). Moreover, let $ X^x $ and $ X^{e_k,x} $ be the solutions of SDE (\ref{MultidimAaffineSDE}) whose parameters are $ ( (\alpha_i)_{i=1,\dots,d},\beta,b) $ and $ ( (\alpha_i)_{i=1,\dots,d},\beta, b + ( \alpha_k / 2 ) e_k ) $ respectively, and their initial values are $ x \in (0,\infty)^d $. Furthermore, assume the above assumptions on \(f\). Then, it holds that
    \[
        \partial_{x_k} E[ f( X^x_t ) ] = e^{ \beta_{k,k} t }E[ (\partial_{y_k}f)( X^{e_k,x}_t ) ]
    \]
\end{corollary}
\begin{proof}[Proof of Theorem \ref{multidimaffinedelta}]
    We will prove the result for the case where $ \beta_{k,k} = 0 $. The proof in the other case is almost the same so we omit it. 

    For simplicity, we consider only the case where $ k=1 $. Then we can rewrite the expectation in the left-hand side of the main result in the theorem as in the one-dimensional case:
    \begin{align*}
        E[ f( X^x_t ) ] &= \int_{\mathbb{R}^d} f(y) \nu_{X^x_t}( dy )    \\
        &=\int_{\mathbb{R}^{d}} \frac{1}{(2\pi)^d} \int_{\mathbb{R}^d} \widehat{f}( \theta ) e^{ i \langle \theta, y \rangle } d\theta\, \nu_{X^x_t}( dy ) \\
        &= \frac{1}{(2\pi)^d} \int_{\mathbb{R}^{d}} \widehat{f}( \theta ) \int_{\mathbb{R}^d} e^{ i \langle \theta, y \rangle } \nu_{X^x_t}( dy ) \, d\theta\ \\
        &= \frac{1}{(2\pi)^d} \int_{\mathbb{R}^{d}} \widehat{f}( \theta ) \varphi_{X^x_t}( \theta ) \, d\theta,
    \end{align*}
    where \( \nu_{X^x_t} \) denotes the distribution of \( X^x_t \), \( \varphi_{X^x_t} \) denotes its characteristic function, and \( \langle \cdot,\cdot \rangle \) denotes the standard inner product on \(\mathbb{R}^d\). 
    
    The remainder of the argument is also straightforward, owing to the simple structure of the SDE. Specifically, \( (X^x)_k \) and \( ((X^x)_1, \dots, (X^x)_{k-1}, (X^x)_{k+1}, \dots, (X^x)_d) \) are independent, \( ((X^{e_k,x})_2,  \dots, (X^{e_k,x})_{k-1}, (X^{e_k,x})_{k+1}, \dots, (X^{e_k,x})_d) \) shares the same distribution as \( ((X^x)_1, \dots, (X^x)_{k-1}, (X^x)_{k+1}, \dots, (X^x)_d) \), and so on.
\end{proof}

The proof of Theorem \ref{multidimaffineIBP} follows the same lines as in the above proof. Similarly, the proof of Corollary \ref{multidimaffinecombined} is obtained along the lines of the proof of Corollary \ref{affinecombined}, so we omit the proofs.

\begin{remark}
    Although the differential equations satisfied by the functions appearing in the characteristic function are known for general multi-dimensional affine diffusions (see e.g. \cite{DFS2003}, Theorem 2.7), their explicit formulas are not available. Therefore, we were not able to derive formulas corresponding to Theorems \ref{multidimaffinedelta}, \ref{multidimaffineIBP}, and Corollary \ref{multidimaffinecombined}. This is the reason why we restrict our attention in this subsection to a particularly simple case. 
\end{remark}

\section{Applications}
\label{sec:5}
\subsection{CIR model}
The class of affine diffusions contains the CIR model, which is a celebrated model for interest rates. The model is given by the following SDE:
\begin{equation}\label{CIR}
    dr_t = k( \theta - r_t ) \,dt + \sigma \sqrt{r_t} \,dW_t.
\end{equation}

Here, $ k, \theta, \sigma \in \mathbb{R}_{>0} $ are parameters, which representing the speed of adjustment to the mean, the mean, and the volatility. We call a process which satisfies the SDE (\ref{CIR}) \emph{a CIR process with parameters \( (\sigma,k,\theta) \)} in this subsection.

In this context, the derivative with respect to the initial value is called \emph{delta} (one of the Greeks) and is used to replicate financial derivatives (See e.g. \cite{Hull}). In view of its importance in application, we rewrite the above results using the symbols of the CIR model.
\begin{corollary}
    Let $ t \in \mathbb{R}_{>0} $ and \( k, \theta, \sigma \in \mathbb{R}_{>0} \). Moreover, let $ r $ and $ r^{ \langle 1 \rangle } $ be CIR processes whose parameters are $ ( \sigma, k, \theta ) $ and $ ( \sigma, k, \theta+\sigma^2/(2k) ) $ respectively, and their initial values are the same positive number. Then, for $ f \in L^1(\mathbb{R}) $ such that \( \widehat{f} \in L^1( \mathbb{R} ) \) and $ f(r_t), f(r^{ \langle 1 \rangle }_t) \in L^1( \Omega ) $, it holds that
    \[
        \partial_{r_0} E[ f( r_t ) ] = \frac{ -2k e^{ -k t } }{ \sigma^2 ( e^{ -k t } -1 ) } \bigl( E[ f( r^{ \langle 1 \rangle }_t) ] - E[ f( r_t ) ] \bigr).
    \]
\end{corollary}
\begin{corollary}\label{CIRIBP}
    Let $ t \in \mathbb{R}_{>0} $ and \( k, \theta, \sigma \in \mathbb{R}_{>0} \) and assume \( 2 k \theta \ge \sigma^2 \). Moreover, let $ r $ and $ r^{ \langle -1 \rangle } $ be CIR whose parameters are $ ( \sigma, k, \theta ) $ and $ ( \sigma, k, \theta-\sigma^2/(2k) ) $ respectively with the same initial positive values. Then, for $ f \in C^1(\mathbb{R}) $ such that $ f $, $ f^\prime $, \( \widehat{f}, \widehat{f^\prime} \in L^1(\mathbb{R}) \) and $ f(r_t), f(r^{ \langle -1 \rangle }_t), f^\prime(r_t) \in L^1( \Omega ) $, it holds that
    \[
        E[ f^\prime( r_t ) ] =\frac{ -2k }{ \sigma^2 ( e^{ -k t } -1 ) } \bigl( E[ f( r_t) ] - E[ f( r^{ \langle -1 \rangle }_t ) ] \bigr).
    \]
\end{corollary}
\begin{corollary}
    Let $ t \in \mathbb{R}_{>0} $ and \( k, \theta, \sigma \in \mathbb{R}_{>0} \). Moreover, let $ r $ and $ r^{ \langle 1 \rangle } $ be CIR processes whose parameters are $ ( \sigma, k, \theta ) $ and $ ( \sigma, k, \theta+\sigma^2/(2k) ) $ respectively with the same initial positive value $ r_0 $. Then, for $ f \in C^1(\mathbb{R}) $ such that $ f, f^\prime, \widehat{f}, \widehat{f^\prime} \in L^1(\mathbb{R}) $ and $ f(r_t), f(r^{ \langle 1 \rangle }_t), f^\prime( r^{ \langle 1 \rangle }_t) \in L^1( \Omega ) $, it holds that
    \[
        \partial_{r_0} E[ f( r_t ) ] = e^{ -kt } E[ f^\prime( r^{ \langle 1 \rangle }_t ) ].
    \]
\end{corollary}

\begin{remark}
     \( 2 k \theta \ge \sigma^2 \), one of the assumptions of Corollary~\ref{CIRIBP}, is known as the \emph{Feller condition}, and it  ensures that the original process \( r \) never hits \(0\).
\end{remark}

\subsection{Model of evolution of population}
The second application concerns a model which is a renormalized limit of  the following branching process which describes the evolution of a population. For details on the following model, see, for example, \cite{Pardoux2016}, Section 3.1.

Let \( X_n \) describe the population at generation $n$ which is defined iteratively as:
\[
    X_{n+1} = \sum_{k=1}^{X_n} \xi_{n,k}, \qquad X_0 = x,
\]
where \(x\) is a positive integer and \( ( \xi_{n,k} )_{n,k} \) are IID random variables which model the descendants of the $k$-th member of the $n$-th generation. 

Now, assume that $ x $ is a positive real number. In order to consider the limit process, we consider $N$ processes of the above type started from a large initial population $[Nx]$ denoted by: \( X^{N,x}, N=1,2,\dots \), where \( X^{N,x}_0 = [Nx] \). These processes may have different offspring distributions. For this reason, we let
\begin{align*}
    \gamma_N &:= N - N E[ \xi^{N}_{0,1} ] \quad\text{i.e.}\quad E[ \xi^{N}_{0,1} ] = 1 + \frac{\gamma_N}{N}, \\
    \sigma^2_N &:= \textrm{Var}( \xi^N_{0,1} ).
\end{align*}

In order to consider the limit population process, we assume
\[
    \gamma_N \to \gamma, \qquad \sigma^2_N \to \sigma^2
\]
for some real numbers \( \gamma \) and \( \sigma^2 > 0 \). Additionally, we assume the following condition:
\[
    E[ |\xi^N_{0,1}|^2 1_{(\xi^N_{0,1} > a \sqrt{N})}]  \to 0 \text{ as } N \to \infty, \text{ for all } a>0.
\]

In this case, the continuous-time process $ \bar{Z}^{N,x}_t := X^{N,x}_{[Nt]} / N $ satisfies the following limit theorem.

\begin{proposition}
    \[
        \bar{Z}^{N,x} \Longrightarrow Z^x,
    \]
    where the convergence is in the Skorohod topology defined 
    in \( D([0,\infty);\mathbb{R}_+) \). Here,  \(Z^x\) solves the SDE
    \[
        d Z^x_t = \gamma Z^x_t dt + \sigma \sqrt{Z^x_t} dB_t,\ t \ge 0,\ Z^x_0 = x. 
    \]
\end{proposition}
The original reference for this result is \cite{Grimvall1974}. In this context, considering the partial derivative of \( E[f(Z^x_t)] \) with respect to \( \gamma \), equation (\ref{beta_E}) can be rewritten as follows:
\begin{align}
    \partial_\gamma E[ f( Z^x_t ) ] = & \frac{ 2 x e^{ \gamma t } }{ \sigma^2 ( e^{ \gamma t } - 1 )^2 } \biggl( ( \gamma t - ( e^{ \gamma t } - 1 ) ) E[ f( Z^{2,x}_t ) ] \notag  \\
    &+ \gamma t ( e^{ \gamma t } - 1 ) E[ f( Z^{1,x}_t ) ] \notag - ( \gamma t - ( e^{ \gamma t } - 1 ) ) E[ f( Z^x_t ) ]  \biggr) \notag
\end{align}
if \( \gamma \neq 0 \) and
\[
    \Bigl. \partial_\gamma E[ f( Z^x_t ) ] \Bigr|_{\gamma=0} = \frac{ x }{ \sigma^2 } \biggl( - E[ f( Z^{2,x}_t ) ] + 2 E[ f( Z^{1,x}_t ) ] - E[ f( Z^x_t ) ] \biggr),
\]
where \( Z^{i,x} \) is the affine diffusion with parameters \( ( \sigma^2, \gamma, i \sigma^2/2 ) \) for \( i=1,2 \).


\begin{thebibliography}{99}
    \bibitem{Alos} Al\`{o}s, E. and Ewald, CO. (2008). Malliavin Differentiability of the Heston Volatility and Applications to Option Pricing. \textit{Adv. Appl. Probab.} \textbf{40}(1), 144--162.

    \bibitem{Alfonsi} Alfonsi, A. (2015). \textit{Affine Diffusions and Related Processes: Simulation, Theory and Applications}. Cham: Springer.
    
    \bibitem{Altman} Altman, H.E. (2018). Bismut--Elworthy--Li Formulae for Bessel Processes. In: Donati-Martin, C., Lejay, A., \& Rouault, A., (Eds.), \textit{S\'{e}minaire de Probabilit\'{e}s XLIX} (pp.~183--220). Cham: Springer.    
    
    \bibitem{DFS2003} Duffie, D., Filipovi\'{c}, D., \& Schachermayer, W. (2003). Affine Processes and Applications in Finance. \textit{Ann. Appl. Probab.} \textbf{13}(3), 984--1053.
    
    \bibitem{Grimvall1974} Grimvall, A. (1974). On the Convergence of Sequences of Branching Processes. \textit{Ann. Probab.} \textbf{2}(6), 1027--1045.

    \bibitem{Hull} Hull, J. (2021). \textit{Options, Futures and Other Derivatives}, 11th edn. London: Pearson.

    \bibitem{JYC2009} Jeanblanc, M., Yor, M., and Chesney, M. (2009). \textit{Mathematical Methods for Financial Markets}. London: Springer.
    
    \bibitem{Pardoux2016} Pardoux, \'E. (2016). \textit{Probabilistic Models of Population Evolution}. Cham: Springer.
\end{thebibliography}
\end{document}